\documentclass[12pt]{article}
\usepackage{amsthm,amsfonts,amsmath,graphicx,amssymb,color}
\newtheorem{theorem}{Theorem}

\newtheorem{lemma}[theorem]{Lemma}

\DeclareMathOperator{\crg}{cr}
\DeclareMathOperator{\lcr}{\overline{lcr}}
\DeclareMathOperator{\lc}{lcr}
\DeclareMathOperator{\rcr}{\overline{cr}}

\definecolor{olivegreen}{rgb}{0.05,0.7,0.1}

\def\tred#1{\textcolor{black}{#1}}

\begin{document}

\title{The rectilinear local crossing number of $K_{n}$}
\author{Bernardo M.~\'{A}brego\\{\small California State University, Northridge}\\{\small bernardo.abrego@csun.edu }
\and Silvia Fern\'{a}ndez-Merchant \thanks{Supported by NSF grant
DMS-1400653}\\{\small California State University,
Northridge}\\{\small silvia.fernandez@csun.edu }} \maketitle

\begin{abstract}
The \emph{local crossing number} of a drawing of a graph is the
largest number of crossings on any edge of the drawing. In a
\emph{rectilinear drawing} of a graph, the vertices are points in
the plane in general position and the edges are straight-line
segments.  The \emph{rectilinear local crossing number} of the
complete graph $ K_n $, denoted by $\lcr(K_n)$, is the minimum local
crossing number over all rectilinear drawings of $ K_n $.

We determine $\lcr(K_n)$. More precisely, for every $n
\notin \{8, 14 \}, $
\[
\lcr(K_n)=\left\lceil \frac{1}{2} \left( n-3-\left\lceil
\frac{n-3}{3} \right\rceil \right) \left\lceil \frac{n-3}{3}
\right\rceil \right\rceil,
\]
$\lcr(K_8)=4$, and $\lcr(K_{14})=15$. \bigskip


\end{abstract}

\section{Introduction}
We are concerned with rectilinear drawings of the complete graph
$K_n$. That is, drawings where each of the $n$ vertices is a point
in the plane, and every edge is drawn as a straight-line segment
among these $n$ vertices. We assume that the set of points is in
general position; that is, there are no three collinear points.

According to Guy et al. \cite{GJS68} and Kainen \cite{Kai73}, Ringel
defined the local crossing number of a graph as follows: In a
\emph{drawing} of a graph, each vertex is represented by a point and
each edge is represented by a simple continuous arc not passing
through any vertex other than its endpoints. The \emph{local
crossing number} of a drawing $D$ of a graph $G$, denoted $\lc (D)$,
is the largest number of crossings on any edge of $D$. The
\emph{local crossing number} of $G$, denoted $\lc(G)$, is the
minimum of $\lc(D)$ over all drawings $D$ of $G$. Other authors
(like Thomassen \cite{Tho88}) have called it the \emph{cross-index},
but we follow Schaefer \cite{Sch14} who strongly encourages the use
of local crossing number. The equivalent definition for rectilinear
drawings is the \emph{rectilinear local crossing number} of $G$,
denoted $\lcr(G)$, as the minimum of $\lc(D)$ over all rectilinear
drawings $D$ of $G$.

Recently, Lara \cite{Lar15} posed the problem of determining $\lcr
(K_n)$. Together with Rubio-Montiel, and Zaragoza, they claimed that
\[
\frac{1}{18}n^2+\Theta(n) \le \lcr(K_n) \le
\frac{1}{9}n^2+\Theta(n),
\]
where the upper bound appears in \cite{LRZ15}.

In the 1960s, Guy, Jenkyns, and Schaer \cite{GJS68} worked on the
problem of determining the local crossing number of $K_n$ drawn on
the surface of a torus. They determined the exact values when $n \le
9$, and they provided asymptotic estimates based on their bounds of
the toroidal crossing number. In particular, they used an argument
in the torus equivalent to the following:  The sum of the number of
crossings of every edge over all edges of a graph $G$ counts
precisely twice the number of crossings of $G$. It follows that
\[
\lcr(G) \ge \frac{2 \rcr (G)}{\tbinom{n}{2}},
\]
where $\rcr (K_n)$ denotes the \emph{rectilinear crossing number} of
$K_n$; that is the smallest number of crossings among all
rectilinear drawings of $K_n$. The current best lower bound by
\'Abrego et al. \cite{ACFLS12}, $\rcr (K_n) \ge \frac{277}{729}
\tbinom{n}{4}+\Theta(n^3)$, yields
\[
\lcr(K_n) \ge \frac{277}{4374} n^2+\Theta(n)
>\frac{1}{15.8}n^2+\Theta(n),
\]
which is already an improvement over the lower bound by Lara et al.
Note that this approach also gives a bound for the local crossing
number of $ K_n $ when the drawings are not necessarily rectilinear.
The current best bound by De Klerk et al. \cite{DeKPS07}$, \crg
(K_n)\geq (0.8594/64)n^4+\Theta (n^3) $, yields
\[
\lc(K_n) \ge \frac{0.8594}{16} n^2+\Theta(n)
>\frac{1}{18.62}n^2+\Theta(n).
\]

In this paper, we determine $\lcr(K_n)$ precisely. Our main result
is the following.

\begin{theorem}\label{th:Main}
If $n$ is a positive integer, then
\begin{align*}
\lcr(K_n) &=\left\lceil \frac{1}{2} \left( n-3-\left\lceil
\frac{n-3}{3} \right\rceil \right) \left\lceil \frac{n-3}{3}
\right\rceil \right\rceil \mbox{if }n\notin \{8,14 \} \text{, that is,}\\
\lcr(K_n)&= \begin{cases} \frac{1}{9}(n-3)^2  &\mbox{if } n \equiv 0 \pmod3, \\
\frac{1}{9}(n-1)(n-4) &\mbox{if } n \equiv 1 \pmod3, \\
\frac{1}{9}(n-2)^2-\left\lfloor \frac{n-2}{6} \right\rfloor
&\mbox{if } n \equiv 2 \pmod3, n \notin \{8,14\}. \end{cases}
\end{align*}
In addition, $\lcr(K_8) =4$ and $\lcr(K_{14})=15$.
\end{theorem}

To prove this theorem, we employ a separation lemma for sets $P$
with $n$ points (Lemma \ref{lem:separation}). We show that there is
either an edge or a path of length 2 whose endpoints are vertices of
the convex hull of $P$ that separates the rest of the set into parts
that differ by at most $(n-2)/3$ or  $(n-3)/3$ points, respectively.
This lemma may have applications to other separation or crossing
problems. To match the lower bound obtained from Lemma
\ref{lem:separation}, we present a couple of constructions: A simple
one that works when $n \not\equiv 2 \pmod 3$ and a more elaborate
construction when $n \equiv 2 \pmod3$.

In Section \ref{sec:lower bound}, we present the lower bound needed
for the proof of Theorem 1. In Sections \ref{sec:firstconstr} and
\ref{sec:secondconstr} we present constructions that match the lower
bound, except for a few exceptional cases. In Section
\ref{sec:MainTheorem} we complete the proof of the Theorem, and in
Section \ref{sec:finalremarks} we present some remarks and open
problems. We also note that the remaining special case of
$\lcr(K_{14}) = 15$ was settled by Aichholzer \cite{A}.

\section{Lower Bound}\label{sec:lower bound}

Let $P$ be a set of points in the plane. For every pair of distinct
points $p$ and $q$ in the plane, denote by $H^{+}(pq)$ the set of
points of $P$ that are on the right side of the oriented line $pq$.
Similarly, define $H^{-}(pq)$ as the set of points of $P$ on the
left side of the oriented line $pq$. If $p$, $q$, and $r$ are three
points in the plane, let $S(prq)$ be the set of points of $P$ in the
open sector defined by the oriented angle $\angle prq$; that is the
sector obtained by rotating a ray counterclockwise around $r$ from
$rp$ to $rq$ (excluding both of the rays $rp$ and $rq$). Finally,
let $ \triangle prq $ denote the interior of triangle $ pqr $. We
first prove the following separation lemma.

\begin{lemma}\label{lem:separation}
If $P$ is a set of $n\ge 3$ points in the plane, then one of the
following occurs:
\begin{itemize}
\item[(a)] There are two vertices $p$ and $q$ of the convex hull of
$P$ such that
\[
\Bigl||H^{+}(pq)|-|H^{-}(pq)|\Bigr| \le \frac{n-2}{3}.
\]

\item[(b)] There are three points $p$, $q$, and $r$ in $P$, such that $p$
and $q$ are vertices of the convex hull of $P$ and
\[
\Bigl||S(prq)|-|S(qrp)|\Bigr| \le \tred{\left\lceil\frac{n-3}{3} \right\rceil}.
\]
\end{itemize}
\end{lemma}

\begin{figure}[htbp]
\begin{center}
\includegraphics[scale=1]{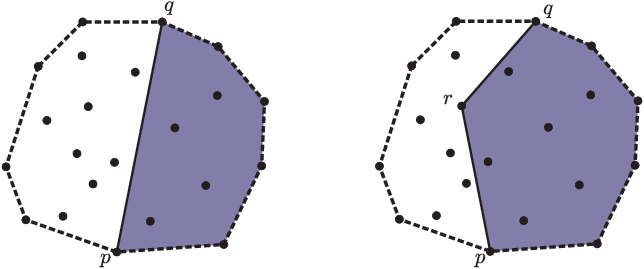}
\caption{The absolute difference of points between the shaded and
not shaded regions is at most $(n-2)/3$ (left) or $\tred{\lceil(n-3)/3\rceil}$
(right).} \label{fig:lemmaIllustr}
\end{center}
\end{figure}
\begin{proof}
The desired situation is illustrated in Figure
\ref{fig:lemmaIllustr}. Suppose that $p_1,p_2,\ldots,p_k$ are the
vertices of the convex hull of $P$ in counterclockwise order. We
first prove that either (a) occurs or else there are $ q_1 $, $ q_2
$, and $ q_3 $ vertices of the convex hull of $ P $ such that $
\triangle q_1q_2q_3 $ contains at least $(n-3)/3$ points of $ P $.

If $(n-2)/3 \le |H^{+}(p_1{p_i})| \le 2(n-2)/3$ for some $ 2\leq i
\leq k $, then
\[
\Bigl||H^{+}(p_1{p_i})|-|H^{-}(p_1{p_i})|\Bigr| \le
\frac{2(n-2)}{3}-\frac{(n-2)}{3} = \frac{n-2}{3},
\]
which is condition (a). So we may assume that for each $ 2\leq i
\leq k$ either $|H^{+}(p_1{p_i})|<(n-2)/3$ or
$|H^{+}(p_1{p_i})|>2(n-2)/3$. Because
\[
0=|H^{+}(p_1p_2)|\leq |H^{+}(p_1p_3)| \leq \ldots \leq
|H^{+}(p_1p_k)|=n-2,
\]
it follows that there is $j$ such that $2 < j \leq k$,
$|H^{+}(p_1p_{j-1})|<(n-2)/3$, and $|H^{+}(p_1p_j)|>2(n-2)/3$. Then
\begin{align*}
|P\cap \triangle
p_1p_{j-1}p_j|&=\left|H^{+}(p_1p_j)\right|-\left|H^{+}(p_1p_{j-1})\right|-1 \\
& \ge \left( \frac{2(n-2)}{3}+\frac{1}{3} \right) - \left(
\frac{n-2}{3}-\frac{1}{3} \right) -1 = \frac{n-3}{3}.
\end{align*}
Thus $\triangle p_1p_{j-1}p_j$ has at least $(n-3)/3$ points of $P$,
which gives the desired triple $ q_1,q_2,q_3$. We assume that the
triangle $ q_1q_2q_3 $ is positively oriented (counterclockwise).

Suppose by contradiction that
\[
\Bigl||S(q_irq_j)|-|S(q_jrq_i)|\Bigr|
> \tred{\left \lceil \frac{n-3}{3}\right \rceil},
\]
for every point $r\in P \setminus \{ q_1,q_2,q_3 \}$ and every $i,j
\in \{1,2,3 \}$ with $i \ne j$. It follows that, for every point $r\in P \setminus \{ q_1,q_2,q_3\}$, two of the following are less than $(n-3)/3$ and the other is greater than $2(n-3)/3$: $|S(q_1rq_2)|$, $|S(q_2rq_3)|$, and $|S(q_3rq_1)|$.

Let $f(r)$ be the pair $(q_i, q_j)$ corresponding to the sector greater than  $2(n-3)/3$. One of the pairs $(q_1, q_2)$, $(q_2,q_3)$,
or $(q_3,q_1)$ must be the image of at least one third of the points
in $ P \setminus \{ q_1,q_2,q_3\} $. Suppose without loss of
generality, that at least $(n-3)/3$ points $r$ have $f(r)=( q_1, q_2
)$. Among all those points $r$, let $r_0$ be one such that
$|S(q_1r_0q_2)|$ is minimum. Note that $r_0$ is in $ H^-(q_1q_2)$,
otherwise $ S(q_1r_0q_2)$ and $\triangle q_1q_2q_3 $ would be
disjoint (Figure \ref{fig:lastpart}a) and together would contain
more than $ 2(n-3)/3+(n-3)/3=n-3 $ points of $P \setminus
\{q_1,q_2,q_3 \}$. Let $p \in S(q_1r_0q_2)$  (Figure
\ref{fig:lastpart}bc). Then $S(q_1pq_2)$ is a proper subset of
$S(q_1r_0q_2)$, and by minimality of $r_0$ it follows that
$|S(q_1pq_2)| \le  2(n-3)/3$. Thus $f(p) \ne (q_1, q_2)$. Therefore
none of the points $p$ in $S(q_1r_0q_2)$ has $f(p)=(q_1,q_2)$; that
is, there are less thant $(n-4)-2(n-3)/3=(n-6)/3$ points $r$ such
that $f(r)=(q_1,q_2)$. This is a contradiction. Hence there is a
point $r$ and a pair $i,j \in \{1,2,3 \}$ with $i \ne j$, such that
\[
\Bigl||S(q_irq_j)|-|S(q_jrq_i)|\Bigr| \le \tred{\left\lceil \frac{n-3}{3} \right\rceil},
\]
and condition (b) is satisfied.
\begin{figure}[htbp]
\begin{center}
\includegraphics[scale=1]{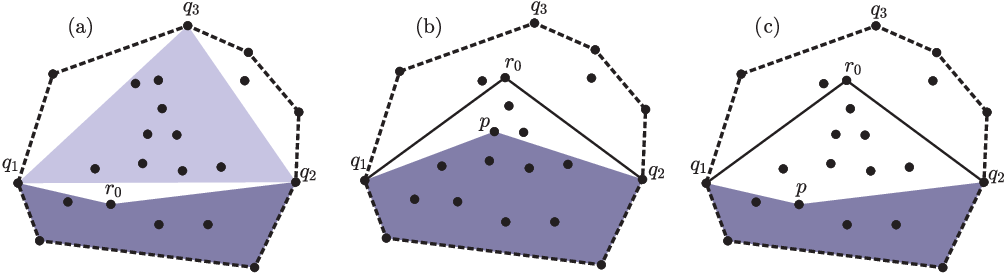}
\caption{(a) $\triangle q_1q_2q_3 \cup S(q_1r_0q_2)$ has more than
$n-3$ points of $P \setminus \{q_1,q_2,q_3 \}$. (bc) $S(q_1pq_2)$ is
a proper subset of $S(q_1r_0q_2)$.} \label{fig:lastpart}
\end{center}
\end{figure}
\end{proof}

Now we prove the lower bound.

\begin{theorem}\label{th:lowerbound}
If $n\ge 3$, then
\[
\lcr(K_n) \ge \begin{cases}  \frac{1}{9}(n-3)^2  &\mbox{if } n \equiv 0 \\
\frac{1}{9}(n-1)(n-4) & \mbox{if } n \equiv 1 \\
\frac{1}{9}(n-2)^2-\lfloor \frac{n-2}{6} \rfloor & \mbox{if } n
\equiv 2.
\end{cases} \pmod{3}
\]
\end{theorem}

\begin{proof}
The result is trivial for $ n=3 $ and 4. Let $n\ge 5$ and consider a
set $P$ of $n$ points in the plane. We use the previous lemma. First
suppose that there are two vertices $p$ and $q$ of the convex hull
of $P$ such that
\[
\Bigl||H^{+}(pq)|-|H^{-}(pq)|\Bigr| \le {\frac{n-2}{3}}.
\]
Every edge $xy$ with $x\in H^{+}(pq)$ and $y \in H^{-}(pq)$
intersects the edge $pq$. Thus $pq$ is crossed at least
$|H^{+}(pq)|\cdot|H^{-}(pq)|$ times. Because the absolute difference
of the factors is at most $(n-2)/3$ and their sum is $ n-2 $, it
follows that for $ n\geq 5 $
\[
|H^{+}(pq)|\cdot|H^{-}(pq)| \ge  \left\lceil \frac{n-2}{3}
\right\rceil \left( n-2-\left\lceil \frac{n-2}{3} \right\rceil
\right) \ge \frac{2}{9}(n-2)^2,
\]
which is larger than the desired bound for any congruence class.

Now, suppose that there are three points $p$, $q$,  and $r$ in $P$
such that $p$ and $q$ are vertices of the convex hull of $P$ and
\[
\Bigl||S(prq)|-|S(qrp)|\Bigr| \le \tred{\left\lceil \frac{n-3}{3} \right\rceil}.
\]
Every edge $xy$ with $x\in S(prq)$ and $y \in S(qrp)$ intersects the
path $pr \cup rq$. Thus the union of the segments $pr$ and $rq$ is
crossed at least $|S(prq)|\cdot|S(qrp)|$ times. Because the absolute
difference of the factors is at most \tred{$\lceil (n-3)/3 \rceil$} and their sum is $
n-3 $, it follows that
\[
|S(prq)|\cdot|S(qrp)| \ge \left\lceil \frac{n-3}{3} \right\rceil
\left( n-3-\left\lceil \frac{n-3}{3} \right\rceil \right).
\]
Therefore one of $pr$ or $rq$ must be crossed at least $\lceil
\frac{1}{2}\lceil {(n-3)}/{3} \rceil ( n-3-\lceil {(n-3)}/{3} \rceil
) \rceil$ times. Hence
\[
\lcr(K_n) \ge \begin{cases}  \frac{1}{9}(n-3)^2  &\mbox{if } n \equiv 0 \\
\frac{1}{9}(n-1)(n-4) & \mbox{if } n \equiv 1 \\
\left\lceil \frac{1}{18}(n-2)(2n-7) \right\rceil =
\frac{1}{9}(n-2)^2-\lfloor \frac{n-2}{6} \rfloor& \mbox{if } n
\equiv 2.
\end{cases} \pmod{3}\qedhere
\]
\end{proof}

\section{The First Construction}\label{sec:firstconstr}

To complete the proof of Theorem \ref{th:Main}, we present
constructions that match the lower bound. The first construction was
originally presented by Lara et al. \cite{LRZ15}, it is based on the
currently best known constructions for the rectilinear crossing
number of $K_n$ by \'Abrego et al. \cite{AF05,ACFLS}  and the
conjectured structure of the optimal sets (see \cite{ACFLS}). Lara
et al. \cite{LRZ15} asymptotically calculated $\lcr (D)$ for the
construction that follows; however, for the purpose of obtaining an
exact formula, we have to carefully calculate the exact local
crossing number of these drawings. Theorem \ref{th:1stConstr} is
tight  for $ n\not\equiv 2 \pmod 3 $. In the next section, we
present a slightly better construction that matches the lower bound
for the class $ n\equiv 2 \pmod 3 $.

\begin{theorem}\label{th:1stConstr} If $n\ge 3$, then
\[
\lcr(K_n) \le \begin{cases}  \frac{1}{9}(n-3)^2  &\mbox{if } n \equiv 0 \\
\frac{1}{9}(n-1)(n-4) & \mbox{if } n \equiv 1 \\
\frac{1}{9}(n-2)^2 & \mbox{if } n \equiv 2.
\end{cases} \pmod{3}
\]
\end{theorem}

\begin{proof}

Consider an arc $C_0$ of circle passing through the points with
coordinates $(1,0)$, $(3,0)$, and $(2,\varepsilon)$, where
$\varepsilon$ is a very small positive real. Let $C_1$ and $C_2$ be
the rotations of $C_0$ with center at the origin and angles $2\pi/3$
and $4\pi/3$, respectively. We set $\varepsilon$ small enough so
that any secant line to the arc $C_0$ separates $C_1$ from $C_2$.

Thus, if $x_0,y_0 \in C_0$, $x_1 \in C_1$, and $x_2\in C_2$, then
the segments $x_0x_1$ and $y_0x_2$ do not cross, and the same occurs
for the corresponding rotations.

The point set $P$ consists of $n_0=\lfloor n/3 \rfloor$ points in
$C_0$, $n_1=\lfloor (n+1)/3 \rfloor$ points in $C_1$, and
$n_2=\lfloor (n+2)/3 \rfloor$ points in $C_2$.

Now we determine $\lcr(P)$. First suppose that $xy$ is an edge with
both endpoints in $C_i$. Moreover, suppose that there are $a$ points
of $P$ in the arc $C_i$ between $x$ and $y$ (see Figure
\ref{fig:construction}(a)). The edges obtained from connecting each
of these $a$ points to the remaining $n_i-a-2$ points on $C_i$, or
to any of the points in $C_{i-1}$ (where the indices are taken
modulo 3) all cross $xy$. By convexity of $C_i$ none of the edges
from $C_i$ to $C_{i+1}$ crosses $xy$, the same is true about edges
from $C_{i-1}$ to $C_{i+1}$. Thus $xy$ is crossed exactly
$a(n_{i-1}+n_i-a-2)$ times. The maximum occurs when $a=n_i-2$ and it
equals $(n_i-2)n_{i-1}$ which is at most $(n_2-2)n_1$.
\begin{figure}[htbp]
\begin{center}
\includegraphics[scale=1]{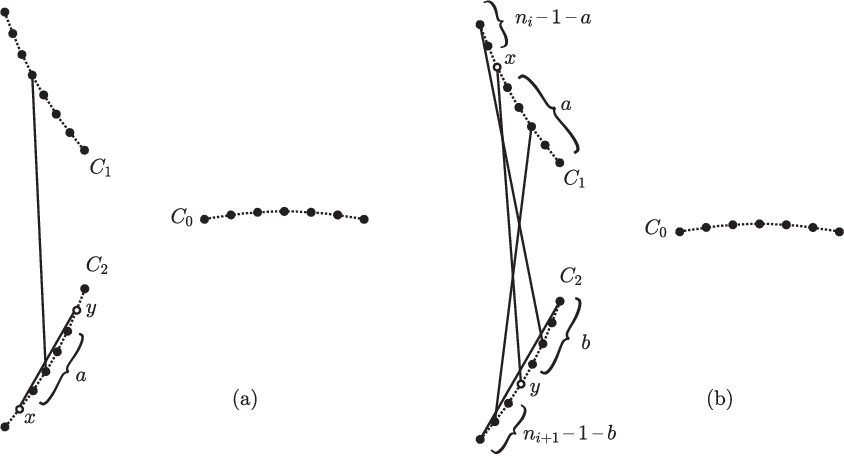}
\caption{The number of crossings of edges in the same $C_i$ or
between $C_i$ and $C_{i+1}$} \label{fig:construction}
\end{center}
\end{figure}

Now, suppose that $xy$ is an edge with $x\in C_i$ and $y \in
C_{i+1}$. Moreover, suppose that there are $a$ points of $P$ before
$x$ in $C_i$ and $n_i-1-a$ after $x$, and $b$ points of $P$ before
$y$ in $C_{i+1}$ and $n_{i+1}-1-b$ after $y$, as shown in Figure
\ref{fig:construction}(b). Each edge $uv$ with $u$ before $x$ in
$C_i$ and $v$ after $y$ in $C_{i+1}$ crosses $xy$. The same occurs
for the edges $uv$ with $u$ after $x$ and $v$ before $y$. In
addition, if $uv$ is an edge in $C_{i+1}$ with $u$ before $y$ and
$v$ after $y$, then $uv$ crosses $xy$. Edges with both endpoints in
$C_{i}$ or both in $C_{i-1}$ do not cross $xy$. The same is true for
edges with one endpoint in $C_{i-1}$ and the other in $C_i$ or
$C_{i+1}$. Thus $xy$ is crossed exactly
\[
b(n_{i+1}-1-b)+b(n_i-1-a)+a(n_{i+1}-1-b)
\]
times. This expression factors as
\[
(n_{i+1}-1)(n_i-1)-a(n_i-1-a)-(a+b+1-n_{i+1})(a+b+1-n_{i}).
\]
Because $0 \le a \le n_i-1$, the second term is nonpositive. The
third term is nonpositive because the two factors are either equal
or differ by one. Therefore, the maximum number of crossings of $xy$
is $(n_{i+1}-1)(n_i-1)$ and it occurs when $a=0$ and $b=n_{i+1}-1$
or $b=n_i-1=n_{i+1}-2$ (only if $n_i < n_{i+1}$); or when $a=n_i-1$
and $b=0$ or $b=n_{i+1}-n_i=1$ (only if $n_i < n_{i+1}$). In any
case, $(n_{i+1}-1)(n_i-1)\le (n_2-1)(n_1-1)$, and comparing to the
first case, $(n_2-2)n_1 \le (n_2-1)(n_1-1)$. Therefore,
\[
\lcr(P)=(n_2-1)(n_1-1)=\begin{cases}  \frac{1}{9}(n-3)^2  &\mbox{if } n \equiv 0 \\
\frac{1}{9}(n-1)(n-4) & \mbox{if } n \equiv 1 \\
\frac{1}{9}(n-2)^2 & \mbox{if } n \equiv 2.
\end{cases} \pmod{3} \qedhere
\]
\end{proof}

\section{The Second Construction}\label{sec:secondconstr}
In this section, we present a construction that matches the lower
bound of $\lcr(K_n)$ when $n \equiv 2 \pmod3 $. Set $n =3k+8$ for
some positive integer $k \ge 4$.

\begin{theorem}\label{th:secondconstr}
If $n=3k+8$ with $k\ge 4$, then
\[
\lcr(K_n) \le \frac{1}{9}(n-2)^2-\left\lfloor \frac{n-2}{6}
\right\rfloor =k^2+4k+3-\left\lfloor \frac{k}{2} \right\rfloor.
\]
\end{theorem}
\begin{figure}[h]
\begin{center}
\includegraphics[scale=.85]{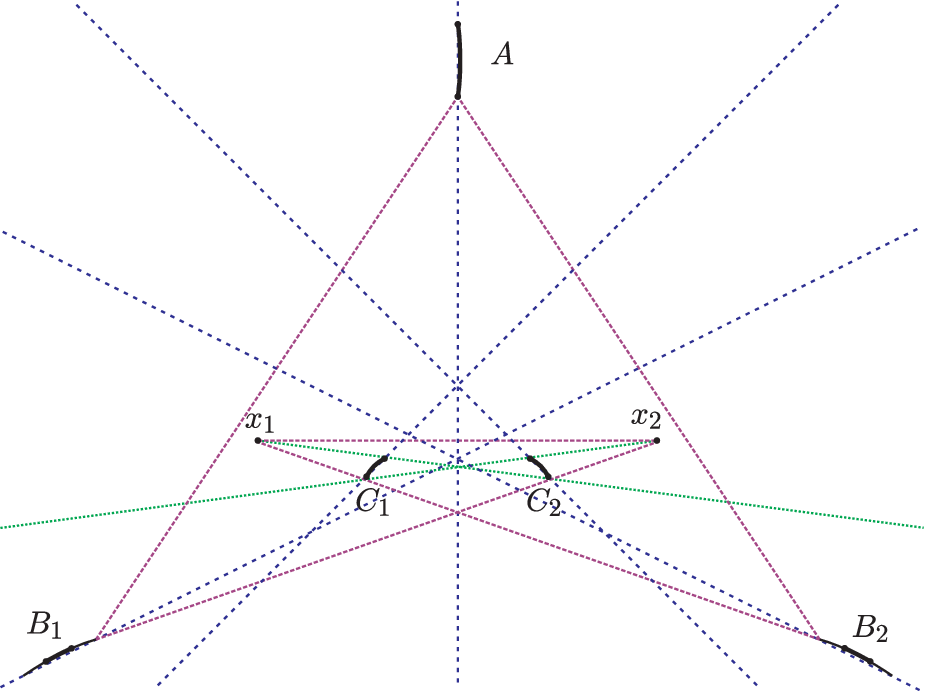}
\caption{The construction when $n=3k+8$. Each of the five thick
dashed lines joins the endpoints of the corresponding arc.}
\label{fig:construction2}
\end{center}
\end{figure}
\begin{proof}
The construction consists of five parts with a linear number of
points and two additional isolated points. Specifically, let
$\varepsilon>0$ be a small constant to be determined. Let $A$ be a
set of $k+2$ points on the arc of circle with endpoints $(0,24)$ and
$(0,20)$ that passes through $(\varepsilon,22)$. Similarly, let
$B_1$ be a set of $k+2-\lfloor k/2 \rfloor$ points on the arc of
circle with endpoints $(-24,-12)$ and $(-20,-10)$ that passes
through $(-22-\varepsilon,-11+2\varepsilon)$; and let $C_1$ be a set
of $\lfloor k/2 \rfloor$ points on the arc of circle with endpoints
$(-5,-1)$ and $(-4,0)$ that passes through
$(-4.5-\varepsilon,-0.5+\varepsilon)$. Let $x_1=(-11,1)$, and set
$B_2$, $C_2$, and $x_2$ as the reflections with respect to the
$y$-axis of $B_1$, $C_1$, and $x_1$, respectively. Choose
$\varepsilon$ small enough so that none of the lines tangent to the
arcs of circle intersects any of the other arcs of circle. This
choice ensures that every line joining points in the same part
separates the other parts as shown in Figure \ref{fig:construction2}
(any $\varepsilon < 0.03$ works). The set $P$ consists of $A\cup B_1
\cup B_2 \cup C_1 \cup C_2 \cup \{x_1,x_2\}$ and it has
$n=(k+2)+2(k+2-\lfloor k/2 \rfloor)+2 \lfloor k/2 \rfloor +2=3k+8$
points.

To determine $\lcr(P)$, we proceed as in Theorem \ref{th:1stConstr};
that is, we find the maximum number of crossings on a segment $xy$
whose endpoints belong to a prescribed pair of parts. Accounting for
symmetry and for the fact that $A$ is on an arc of circle that is
concave to the left, the cases that need to be analyzed are the
following: $(A,A)$, $(A,\{x_1\})$, $(A,B_1)$, $(A,C_1)$, $(\{x_1\},\{x_2\})$,
$(\{x_1\},B_1)$, $(\{x_1\},B_2)$, $(\{x_1\},C_1)$, $(\{x_1\},C_2)$, $(B_1,B_1)$,
$(B_1,B_2)$, $(B_1,C_1)$, $(B_1,C_2)$, $(C_1,C_1)$, and $(C_1,C_2)$;
where each pair $(X,Y)$ indicates that $x \in X$ and $y \in Y$. In
all the following cases with $X \ne Y$, if $p\in \{x,y \}$ belongs
to $S \in \{A,B_1,B_2,C_1,C_2 \}$, then we call \emph{lower} $S$ the
set of points in $S$ that have $y$-coordinate less than $p$, and we
call \emph{upper} $S$ the set of points in $S$ that have
$y$-coordinate greater than $p$.

\begin{itemize}

\item[$(A,A)$] Let $x,y\in A$. Suppose that there are $a$ points of $A$ between $x$ and $y$, and
$k-a$ outside. The only segments crossing $xy$ are those connecting
one of the $a$ inside points in $A$ to points in $C_1 \cup B_1 \cup
\{x_1 \}$ or to points in $A$ outside $xy$. The total is
\[
a\left(\left\lfloor \tfrac{k}{2}
\right\rfloor+\left(k+2-\left\lfloor \tfrac{k}{2} \right\rfloor
\right)+1+(k-a)\right)=a(2k+3-a).
\]
This is maximized when $a =k$ and it gives $k^2+3k$.

\item[$(A,\{x_1\})$] Let $x=x_1$ and $y\in A$. Suppose that upper $A$
and lower $A$ have $a$ and $k+1-a$ points of $A$, respectively. The
only segments crossing $xy$ are those connecting (upper $A$) to
$C_1\cup (\text{lower }A)$, or (lower $A$) to $B_1$. The total is
\begin{multline*}
a\left(\left\lfloor \tfrac{k}{2}
\right\rfloor+(k+1-a)\right)+(k+1-a)\left(k+2-\left\lfloor
\tfrac{k}{2} \right\rfloor \right)\\
=a(2 \left\lfloor \tfrac{k}{2}
\right\rfloor-1-a)+(k+1)(k+2-\left\lfloor \tfrac{k}{2}
\right\rfloor).
\end{multline*}
This is maximized when $a \in \{\lfloor k/2 \rfloor, \lfloor k/2
\rfloor -1 \}$ and it gives $k^2+3k+2-\lfloor k/2 \rfloor
(k+2-\lfloor k/2 \rfloor)<k^2+3k+2$.

\item[$(A,B_1)$] Let $x \in B_1$ and $y \in A$. Suppose that upper $A$, lower $A$, lower
$B_1$, and upper $B_1$ have $a$, $k+1-a$, $b$, and $k+1-\lfloor k/2
\rfloor-b$ points, respectively. The only segments crossing $xy$ are
those connecting (upper $A$) to $C_1\cup (\text{lower }A)\cup
(\text{upper }B_1) \cup \{ x_1 \}$, or (lower $B_1$) to
$(\text{lower }A)\cup C_1 \cup \{x_1\}$. The total is
\begin{multline*}
a[\left\lfloor \tfrac{k}{2} \right\rfloor + (k+1-a)+
(k+1-\left\lfloor \tfrac{k}{2} \right\rfloor -b)
+1]+b[(k+1-a)+\left\lfloor \tfrac{k}{2} \right\rfloor+1]\\
=a(2k+3-2b-a)+b(k+2+\left\lfloor \tfrac{k}{2} \right\rfloor).
\end{multline*}
As a function of $a$, this is maximized when $a \in \{ k+1-b,k+2-b
\}$ and it gives
\[
b(b-k-1+\left\lfloor \tfrac{k}{2} \right\rfloor)+(k^2+3k+2).
\]
This is maximized when $b \in \{ 0, k+1-\lfloor {k}/{2} \rfloor \}$
and it gives $k^2+3k+2$.

\item[$(A,C_1)$] Let $x \in C_1$ and $y \in A$. Suppose that upper $A$, lower $A$, lower
$C_1$, and upper $C_1$ have $a$, $k+1-a$, $c$, and $\lfloor k/2
\rfloor-1-c$ points, respectively. The only segments crossing $xy$
are those connecting (upper $A$) to $(\text{upper }C_1)\cup
(\text{lower }A)$, or (lower $A$) to $B_1\cup (\text{lower }C_1)
\cup \{x_1\}$, or (upper $C_1$) to $B_1 \cup \{ x_1 \}$, or $x_1$ to
$C_2 \cup \{x_2\}$. The total is
\begin{multline*}
a[(\left\lfloor \tfrac{k}{2}
\right\rfloor-1-c)+(k+1-a)]+(k+1-a)[(k+2-\left\lfloor \tfrac{k}{2}
\right\rfloor)+c+1]\\
+(\left\lfloor \tfrac{k}{2} \right\rfloor-1-c)(k+2-\left\lfloor
\tfrac{k}{2} \right\rfloor+1)+(\left\lfloor \tfrac{k}{2}
\right\rfloor+1)\\
=a(2\left\lfloor \tfrac{k}{2} \right\rfloor-3-2c-a)+c(\left\lfloor
\tfrac{k}{2} \right\rfloor-2)+k^2+3k+1+4\left\lfloor \tfrac{k}{2}
\right\rfloor-\left\lfloor \tfrac{k}{2} \right\rfloor^2.
\end{multline*}
As a function of $a$, this is maximized when $a \in \{ \lfloor k/2
\rfloor -c-1,\lfloor k/2 \rfloor -c-2 \}$ and it gives
\[
c(c+1-\left\lfloor \tfrac{k}{2}
\right\rfloor)+(k^2+3k+3+\left\lfloor \tfrac{k}{2} \right\rfloor).
\]
This is maximized when $c \in \{ 0, \lfloor {k}/{2} \rfloor -1 \}$
and it gives $k^2+3k+3+\lfloor k/2 \rfloor$.

\item[$(\{x_1\},\{x_2\})$] The only segments that cross $x_1x_2$ are those
with one endpoint in $A$ and the other in $C_1 \cup C_2$. The total
is $(k+2)(2 \lfloor k/2 \rfloor)\le k^2+2k$.

\item[$(\{x_1\},B_1)$] Let $x=x_1$ and $y \in B_1$. Suppose that lower
$B_1$ has $b$ points and upper $B_1$ has $k+1-\lfloor k/2 \rfloor-b$
points. The only segments crossing $xy$ are those connecting lower
$B_1$ to $C_1$, or upper $B_1$ to $A$. The total is
\[
b \left \lfloor \tfrac{k}{2} \right\rfloor  +(k+1-\left\lfloor
\tfrac{k}{2} \right\rfloor -b)(k+2) =b(\left \lfloor \tfrac{k}{2}
\right\rfloor -k-2)+(k+2)(k+1-\left \lfloor \tfrac{k}{2}
\right\rfloor ).
\]
This is maximized when $b=0$ and it gives $(k+2)(k+1- \lfloor
{k}/{2} \rfloor)< k^2+3k+2$.

\item[$(\{x_1\},B_2)$] Let $x=x_1$ and $y \in B_2$. Suppose that lower
$B_2$ has $b$ points and upper $B_2$ has $k+1-\lfloor k/2 \rfloor-b$
points. The only segments crossing $xy$ are those connecting $B_1$
to $C_1 \cup C_2 \cup (\text{upper }B_2)\cup \{x_2 \}$, or
$(\text{lower }B_2)$ to $C_1 \cup (\text{upper }B_2)$. The total is
\begin{multline*}
(k+2- \left \lfloor \tfrac{k}{2} \right \rfloor)\left[\left\lfloor
\tfrac{k}{2} \right\rfloor + \left\lfloor \tfrac{k}{2}
\right\rfloor+(k+1-\left\lfloor \tfrac{k}{2}
\right\rfloor-b)+1\right]+b \left[\left\lfloor \tfrac{k}{2}
\right\rfloor+(k+1-\left\lfloor
\tfrac{k}{2} \right\rfloor -b) \right]\\
=b(\left\lfloor \tfrac{k}{2} \right\rfloor
-1-b)+(k^2+4k+4)-\left\lfloor \tfrac{k}{2} \right\rfloor^2\\
<b(\left\lfloor \tfrac{k}{2} \right\rfloor
-b)+(k^2+4k+4)-\left\lfloor \tfrac{k}{2} \right\rfloor^2 .
\end{multline*}
This is maximized (for $b$ real) when $b=\lfloor k/2 \rfloor/2$ and
it gives $k^2+4k+4-\tfrac{3}{4} \lfloor {k}/{2} \rfloor^2$, which is
at most $k^2+4k+3-\lfloor {k}/{2} \rfloor$ for $k \ge 4$.

\item[$(\{x_1\},C_1)$] Let $x=x_1$ and $y \in C_1$. Suppose that lower
$C_1$ has $c$ points and upper $C_1$ has $\lfloor k/2 \rfloor-1-c$
points. The only segments crossing $xy$ are those connecting upper
$C_1$ to $B_1$, or lower $C_1$ to $A$. The total is
\[
(\left\lfloor \tfrac{k}{2} \right\rfloor-1-c)(k+2-\left\lfloor
\tfrac{k}{2} \right\rfloor)+c(k+2)=(\left\lfloor \tfrac{k}{2}
\right\rfloor-1)(k+2-\left\lfloor \tfrac{k}{2}
\right\rfloor)+c\left\lfloor \tfrac{k}{2} \right\rfloor.
\]
This is maximized when $c=\lfloor {k}/{2} \rfloor-1$ and it gives
$(k+2)(\lfloor {k}/{2} \rfloor-1) \le k^2/2-2$.

\item[$(\{x_1\},C_2)$] Let $x=x_1$ and $y \in C_2$. Suppose that lower
$C_2$ has $c$ points and upper $C_2$ has $\lfloor k/2 \rfloor-1-c$
points. Note that there is a line through $ x_1 $ that separates $
C_1 $ and $ C_2 $. Then the only segments crossing $xy$ are those
connecting $C_1$ to $A \cup \{x_2 \}$ or (upper $C_2$) to $B_1 \cup
C_1 \cup (\text{lower }C_2)$. The total is
\begin{multline*}
\left\lfloor \tfrac{k}{2} \right\rfloor
\left[(k+2)+1\right]+(\left\lfloor \tfrac{k}{2}
\right\rfloor-1-c)\left[(k+2-\left\lfloor \tfrac{k}{2}
\right\rfloor)+\left\lfloor \tfrac{k}{2}
\right\rfloor+c\right]\\
=c(\left\lfloor \tfrac{k}{2}
\right\rfloor-k-3-c)+(k+2)(2\left\lfloor \tfrac{k}{2}
\right\rfloor-1)+\left\lfloor \tfrac{k}{2} \right\rfloor.
\end{multline*}
This is maximized when $c=0$ and it gives $(k+2)(2\lfloor {k}/{2}
\rfloor-1)+\lfloor {k}/{2} \rfloor \le k^2+3k/2-2$.

\item[$(B_1,B_1)$] Let $xy$ be a segment with both points in $B_1$.
Suppose that there are $b$ points of $B_1$ between $x$ and $y$, and
$k-\lfloor k/2 \rfloor -b$ outside. The only segments crossing $xy$
are those connecting one of the $b$ inside points in $B_1$ to points
in $B_2 \cup C_2 \cup \{x_2 \}$ or to points in $B_1$ outside $xy$.
The total is
\[
b\left[(k+2-\left\lfloor \tfrac{k}{2} \right\rfloor)+\left\lfloor
\tfrac{k}{2} \right\rfloor+1+\left(k-\left\lfloor \tfrac{k}{2}
\right\rfloor-b \right)\right]=b[2k+3-\left\lfloor \tfrac{k}{2}
\right\rfloor-b].
\]
This is maximized  when $b= k-\lfloor {k}/{2}
\rfloor$ and it gives $( k-\lfloor {k}/{2}
\rfloor)(k+3)\leq k^2+3k$ for $k \ge 4$.

\item[$(B_1,B_2)$] Let $x \in B_1$ and $y \in B_2$. Suppose that lower $B_1$, upper $B_1$,
lower $B_2$, and upper $B_2$ have $b_1$, $k+1-\lfloor k/2
\rfloor-b_1$, $b_2$, and $k+1-\lfloor k/2 \rfloor-b_2$ points,
respectively. The only segments crossing $xy$ are those connecting
(lower $B_1$) to $(\text{upper }B_2)\cup (\text{upper }B_1) \cup C_2
\cup \{x_2 \}$, or (lower $B_2$) to $(\text{upper }B_1)\cup
(\text{upper }B_2) \cup C_1 \cup \{x_1\}$. The total is
\begin{multline*}
b_1[(k+1-\left\lfloor \tfrac{k}{2}
\right\rfloor-b_2)+(k+1-\left\lfloor \tfrac{k}{2}
\right\rfloor-b_1)+\left\lfloor \tfrac{k}{2}
\right\rfloor+1]\\
+b_2[(k+1-\left\lfloor \tfrac{k}{2}
\right\rfloor-b_1)+(k+1-\left\lfloor \tfrac{k}{2}
\right\rfloor-b_2)+\left\lfloor \tfrac{k}{2}
\right\rfloor+1]\\
=(b_1+b_2)(2k+3-\left\lfloor \tfrac{k}{2}
\right\rfloor-(b_1+b_2))\\
<(b_1+b_2)(2k+3-(b_1+b_2)).
\end{multline*}
As a function of $b_1+b_2$, this is maximized when $b_1+b_2 \in
\{k+1,k+2 \}$ and it gives $k^2+3k+2$.

\item[$(B_1,C_1)$] Let $x \in B_1$ and $y \in C_1$.
Suppose that lower $B_1$, upper $B_1$, lower $C_1$, and upper $C_1$
have $b$, $k+1-\lfloor k/2 \rfloor-b$, $\lfloor k/2 \rfloor-1-c$,
and $c$ points, respectively. The only segments crossing $xy$ are
those connecting (lower $C_1$) to $(\text{lower }B_1)\cup A \cup
\{x_1 \}$, or (upper $B_1$) to $(\text{upper }C_1)\cup A \cup
\{x_1\}$, or $x_1$ to $B_2$. The total is
\begin{multline*}
(\left\lfloor \tfrac{k}{2}
\right\rfloor-1-c)[b+(k+2)+1]+(k+1-\left\lfloor \tfrac{k}{2}
\right\rfloor-b)[c+(k+2)+1]+(k+2-\left\lfloor \tfrac{k}{2}
\right\rfloor)\\
=-b(2c+4+k-\left\lfloor \tfrac{k}{2} \right\rfloor)-c(2+\left\lfloor
\tfrac{k}{2} \right\rfloor)+k^2+4k+2-\left\lfloor \tfrac{k}{2}
\right\rfloor.
\end{multline*}
As a function of $b$ and $c$, this is maximized when $b=c=0$ and it
gives $k^2+4k+2-\lfloor k/2 \rfloor$.

\item[$(B_1,C_2)$] Let $x \in B_1$ and $y \in C_2$. Suppose that lower $B_1$, upper $B_1$, lower
$C_2$, and upper $C_2$ have $b$, $k+1-\lfloor k/2 \rfloor-b$,
$\lfloor k/2 \rfloor-1-c$, and $c$ points, respectively. The only
segments crossing $xy$ are those connecting (lower $B_1$) to
$(\text{upper }C_2) \cup (\text{upper }B_1)$, or (upper $B_1$) to
$B_2 \cup (\text{lower }C_2)\cup \{x_2 \}$, or $B_2 \cup
(\text{lower }C_2)$ to $C_1\cup \{x_1\}$, or (upper $C_2$) to (lower
$C_2$). The total is
\begin{multline*}
b[c+(k+1-\left\lfloor \tfrac{k}{2}
\right\rfloor-b)]+(k+1-\left\lfloor \tfrac{k}{2}
\right\rfloor-b)[(k+2-\left\lfloor \tfrac{k}{2}
\right\rfloor)+(\left\lfloor \tfrac{k}{2}
\right\rfloor-1-c)+1]\\
+[(k+2-\left\lfloor \tfrac{k}{2}
\right\rfloor)+(\left\lfloor \tfrac{k}{2}
\right\rfloor-1-c)](\left\lfloor \tfrac{k}{2} \right\rfloor+1)
+c(\left\lfloor \tfrac{k}{2} \right\rfloor-1-c)\\
=-(b-c)^2-b(\left\lfloor \tfrac{k}{2}
\right\rfloor+1)-c(k-\left\lfloor \tfrac{k}{2}
\right\rfloor+3)+k^2+4k+3-\left\lfloor \tfrac{k}{2} \right\rfloor.
\end{multline*}
As a function of $b$ and $c$, this is maximized when $b=c=0$ and it
gives $k^2+4k+3-\lfloor k/2 \rfloor$.

\item[$(C_1,C_1)$] Let $xy$ be a segment with both points in $C_1$.
Suppose that there are $c$ points of $C_1$ between $x$ and $y$, and
$\lfloor k/2 \rfloor -2-c$ outside. The only segments crossing $xy$
are those connecting one of the $c$ inside points in $C_1$ to points
in $B_2 \cup C_2 \cup \{x_2 \}$ or to points in $C_1$ outside $xy$.
The total is
\[
c\left[(k+2-\left\lfloor \tfrac{k}{2} \right\rfloor)+\left\lfloor
\tfrac{k}{2} \right\rfloor+1+\left(\left\lfloor \tfrac{k}{2}
\right\rfloor-2-c \right)\right]=c[k+\left\lfloor \tfrac{k}{2}
\right\rfloor+1-c].
\]
This is maximized when $c= \lfloor {k}/{2} \rfloor-2$ and it gives
$(k+3)(\lfloor {k}/{2} \rfloor -2)<k^2/2$.

\item[$(C_1,C_2)$] Let $x \in C_1$ and $y \in C_2$. Suppose that lower $C_1$, upper $C_1$,
lower $C_2$, and upper $C_2$ have $\lfloor k/2 \rfloor-1-c_1$,
$c_1$, $\lfloor k/2 \rfloor-1-c_2$, and $c_2$ points, respectively.
The only segments crossing $xy$ are those connecting (upper $C_1$)
to $B_2 \cup (\text{lower }C_2) \cup (\text{lower }C_1)$, or (upper
$C_2$) to $B_1 \cup (\text{lower }C_1 \cup (\text{lower }C_2)$, or
$x_1$ to (lower $C_2$), or $x_2$ to (lower $C_1$). The total is
\begin{multline*}
c_1[(k+2-\left\lfloor \tfrac{k}{2} \right\rfloor)+(\left\lfloor
\tfrac{k}{2} \right\rfloor-1-c_2)+(\left\lfloor \tfrac{k}{2}
\right\rfloor-1-c_1)]\\
+c_2[(k+2-\left\lfloor \tfrac{k}{2} \right\rfloor)+(\left\lfloor
\tfrac{k}{2} \right\rfloor-1-c_1)+(\left\lfloor \tfrac{k}{2}
\right\rfloor-1-c_2)]\\
+(\left\lfloor \tfrac{k}{2} \right\rfloor-1-c_1)+(\left\lfloor
\tfrac{k}{2}
\right\rfloor-1-c_2)\\
=(c_1+c_2)(k+\left\lfloor \tfrac{k}{2}
\right\rfloor-1-c_1-c_2)+2\left\lfloor \tfrac{k}{2}
\right\rfloor-2\\
<(c_1+c_2)[2k-(c_1+c_2)]+k-2.
\end{multline*}
As a function of $c_1+c_2$, this is maximized when $c_1+c_2=k$ and
it gives $k^2+k-2$.
\end{itemize}
In all cases the maximum possible number of crossings is at most
$k^2+4k+3-\lfloor k/2 \rfloor$. Equality occurs in the case $(B_1,
C_2)$ when $b=c=0$, and for $k$ even in the case $(A,C_1)$ when
$(a,c)\in \{(0,k/2-1),(k/2-1,0),(k/2-2,0) \}$; together with the
symmetric cases $(B_2,C_1)$ and $(A,C_2)$.
\end{proof}

\section{Proof of Theorem \ref{th:Main}}\label{sec:MainTheorem}

If $n\le 2$, then the result is
trivially true because $\lcr(K_n)=0$. Suppose $n\ge 3$. The required
formula as lower bound follows directly by Theorem
\ref{th:lowerbound}. If $n\equiv 0,1 \pmod 3$, then the matching
upper bound follows from Theorem \ref{th:1stConstr}. If $n\equiv 2
\pmod 3$ and $n\ge 20$, then the matching upper bound follows from
Theorem \ref{th:secondconstr}. The upper bounds $\lcr(K_5)\le 1$ and
$\lcr(K_8)\le 4$ follow from Theorem \ref{th:1stConstr}. The
drawings shown in Figure \ref{fig:constr11-14-17} show that
$\lcr(K_{11})\le 8$, $\lcr(K_{14})\le 15$, and $\lcr(K_{17})\le 23$.
Finally, for $n=8$, we used the data base by Aichholzer et al. \cite{aak1,oswinweb2} to calculate the local crossing number of each of the 3315 order types of 8 points. Our calculations verified that $\lcr(K_8) \ge 4$ for all of them, with equality achieved by 39 order types. For $n=14$, Aichholzer \cite{A} extended his data base from $n=11$ to $n=14$ for this specific problem. He confirmed our results for $n<14$ and verified our earlier conjecture \cite{AF16} that there are no geometric sets of 14 points where every edge is crossed at most 14 times. Thus $\lcr(K_{14})=15$.
\begin{figure}[htbp]
\begin{center}
\includegraphics[scale=1.2]{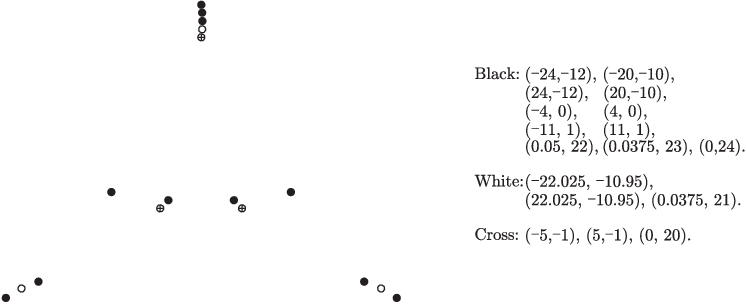}
\caption{The construction for $n\in \{ 11, 14, 17 \}$. The 11-point set consists of the black points, the 14-point set consists of the black or white points, and the 17-point set consists of all of the points shown.} \label{fig:constr11-14-17}
\end{center}
\end{figure}
\section{Final Remarks}\label{sec:finalremarks} The construction in Theorem
\ref{fig:construction} shows that Lemma \ref{lem:separation} is best
possible. That is, the point-sets $P$ constructed in Theorem
\ref{fig:construction} verify that every path of length 2 whose
endpoints are vertices of the convex hull of $P$ separates the rest
of the set into parts that differ by no less than $(n-3)/3$.

Another consequence of Theorem \ref{th:Main} is that for $n \not\in
\{ 8,14 \}$, there is always an edge achieving $\lcr(K_n)$ with one
endpoint that is a vertex of the convex hull. The same is true for
$n=8$ by inspecting all order types that achieve $\lcr(K_8)=4$.
However, it is not true that \emph{all} of the maximal edges on the
optimal examples have an endpoint that is a vertex of the convex
hull. Another property shared by all the optimal constructions we
know is that their convex hull is a triangle. Is it true that all
the optimal point sets have triangular convex hulls? A related
problem is to determine the maximum number of edges that can achieve
$\lcr(K_n)$ crossings in an optimal point set.

\paragraph{Acknowledgments.} We warmly thank the referees for their comments and suggestions,
which certainly improved the presentation of this paper. We also
thank Oswin Aichholzer for listening to our results, settling the
case $ n=14 $, and verifying Theorem \ref{th:Main} for small values
of $ n $. S.~Fern\'andez-Merchant's research was supported by the
NSF grant DMS-1400653.

\end{document}